\documentclass[12pt]{article}
\usepackage{color,bbold}
\usepackage{amssymb,amsmath,amsthm,graphicx, color,amsfonts,tabularx}
\usepackage[english]{babel}
\usepackage{afterpage}
\usepackage{authblk, relsize}
\usepackage{float}
\usepackage{geometry}
\def\supp{\mathop{\rm supp}\nolimits}
\def\lip{\mathop{\rm Lip}\nolimits}
\newtheorem{theorem}{Theorem}[section]

\newtheorem{lemma}[theorem]{Lemma}
\newtheorem{proposition}[theorem]{Proposition}
\newtheorem{corollary}[theorem]{Corollary}
\newtheorem{definition}[theorem]{Definition}
\newtheorem{remark}[theorem]{Remark}
\newtheorem{example}[theorem]{Example}
\newtheorem{notation}[theorem]{Notation}

\renewenvironment{proof}[1][.]{%
	\bigskip\noindent{\bf Proof#1 }}{%
	\hfill$\blacksquare$\bigskip}

\begin{document}
\pagestyle{myheadings}
\title{Optimal transportation and pressure at zero temperature}
\author{Jairo K. Mengue}
\affil{Universidade Federal do Rio Grande do Sul}

\date{\today}
\maketitle

\begin{abstract}
Given two compact metric spaces $X$ and $Y$, a Lipschitz continuous cost function $c$ on $X \times Y$ and two probabilities $\mu \in\mathcal{P}(X),\,\nu\in\mathcal{P}(Y)$, we propose to study the Monge-Kantorovich problem and its duality from  a zero temperature limit of a convex pressure function. We consider the entropy defined by $H(\pi) = -D_{KL}(\pi|\mu\times \nu)$, where $D_{KL}$ is the Kullback-Leibler divergence, and then the pressure defined by the variational principle 
\[P(\beta A) = \sup_{\pi \in \Pi(\mu,\nu)} \left[ \smallint  \beta A\,d\pi + H(\pi)\right],\]where $\beta>0$ and $A=-c$.
We will show that it admits a dual formulation and when $\beta \to+\infty$ we recover the solution for the usual Monge-Kantorovich problem and its Kantorovich duality. Such approach is similar to one which is well known in Thermodynamic Formalism and Ergodic Optimization, where $\beta$ is interpreted as the inverse of the temperature ($\beta = \frac{1}{T}$) and $\beta\to+\infty$ is interpreted as a zero temperature limit.
\end{abstract}

\vspace {.8cm}
  
\noindent 
\emph{Key words and phrases:} \newline Monge-Kantorovich problem, Kantorovich duality, Kullback-Leibler divergence, Entropy, Pressure. 

\noindent
\emph{2020 Mathematics Subject Classification:} \newline Primary:  37A50, 28A33, 28D20, 46E27, 60B10, 60F10, 47H10.


\section{Introduction}

We consider  two compact metric spaces $(X,d_X)$, $(Y,d_Y)$ and the product space $X\times Y$ with the metric $d_X+d_Y$. We consider also the Borel sigma algebra over each space and two probabilities, $\mu \in\mathcal{P}(X)$ and $\nu\in \mathcal{P}(Y)$ such that $\supp(\mu)=X$ and $\supp(\nu)=Y$. A probability $\pi \in \mathcal{P}(X\times Y)$ will be called a \textit{transference plan} if it has $x-$marginal $\mu$ and $y-$marginal $\nu$, that is,
\[\smallint  f(x)\,d\pi(x,y) = \smallint  f(x)\,d\mu(x) \,\,\,\text{and}\,\,\,\smallint  g(y)\,d\pi(x,y) = \smallint  g(y)\,d\nu(y)\]
for any $f\in C(X)$ and $g\in C(Y)$. We denote by $\Pi(\mu,\nu)$ the set of transference plans. 
Let $c:X\times Y\to\mathbb{R}$ be a Lipschitz continuous function (a cost function) and consider the problem  of  to find the number
\[  \alpha(c) = \min_{\pi \in \Pi(\mu,\nu)} \smallint  c(x,y)\,d\pi\]
and an optimal transference plan (Monge-Kantorovich problem).

A classical result concerning such variational principle is Kantorovich Duality: Let $\Phi_c:=\{(\varphi,\psi)\in C(X)\times C(Y)\,|\, c(x,y)\geq \varphi(x)+\psi(y) \,\forall x,y\}$, then
\[\min_{\pi \in \Pi(\mu,\nu)} \smallint  c(x,y)\,d\pi = \sup_{ (\varphi,\psi)\in \Phi_c} \smallint  \varphi(x)\,d\mu(x)+\smallint  \psi(y)\,d\nu(y).\]
Such duality can be formulated in a  broader sense (instead compact spaces and Lipschitz functions) and in \cite{Vi2,Vi1} can be found  a general  exposition of the optimal transportation theory. Anyway we remark that for any $p\geq 1$, if $X=Y$ and the cost function is of the form $c(x,y) = [d(x,y)]^p$ then $c$ is Lipschitz continuos.

As the set $\Pi(\mu,\nu)$ is convex and the map $\pi \mapsto \smallint  c(x,y)\,d\pi$ is affine, the Monge-Kantorovich problem can be interpreted as a variational principle in convex analysis. 
In the present work we propose to consider the Kullback-Leibler divergence and relative entropy and then a different variational principle with introduction  of such entropy (which we call the positive temperature approach). Furthermore, we will recover a solution of the Monge-Kantorovich problem and Kantorovich dual problem as a zero temperature limit, similarly as occur in Thermodynamic Formalism and Ergodic Optimization (see \cite{BLL}). We believe that section 2 of \cite{MO3} also helps to clarify our strategy.  
 
If $\eta$ and $\rho$ are probabilities on a same measurable space, the Kullback-Leibler divergence is known as 
\[D_{KL}(\eta|\rho)  :=\left\{\begin{array}{cc}
	\int  \log(\frac{d\eta}{d\rho})\,d\eta & \text{if}\, \eta \ll \rho\\ \\
	+\infty & \begin{array}{c}\text{if}\,\eta\,\text{is not absolutely continuous}\\\text{with respect to}\,\rho \end{array}\end{array}\right. ,  \]
where $\frac{d\eta}{d\rho}$ denotes the Radon-Nikodym derivative.

Following \cite{LM3} we define the entropy of a probability $\eta \in \mathcal{P}(X\times Y)$ relative to $\mu \in \mathcal{P}(X)$ by 
$H^\mu(\eta)=-D_{KL}(\eta\,|\, \mu\times \rho)$, where $\rho$ is the $y-$marginal of $\eta$.  We refer \cite{LM3} for additional results and characterization of $H^\mu$.  In  \cite{LM3,LMMS2, MO,MO3} can be found such entropy in the context of Ruelle Operator and a variational principle of the pressure, which is defined from a dynamical system or an iterated function system (IFS).

Supposing $\pi \in \Pi(\mu,\nu)$ we have  $H(\pi):=H^\mu(\pi)=-D_{KL}(\pi\,|\, \mu\times \nu)$.
In the present work, using such entropy $H(\pi)$,  we introduce the following definition of pressure:
\[P(A) = \sup_{\pi \in \Pi(\mu,\nu)} \left[ \smallint   A\,d\pi + H(\pi)\right],\]
where $A:X\times Y \to \mathbb{R}$ is continuous. Let us consider the problem of to find a transference plan  which attains the supremum for $P$ (an equilibrium).  In \cite{LM, LMMS2, MO} appear some duality results for variational principles in different settings which consider a mixing of transport theory and ergodic theory. In such works it is considered a dynamic or IFS, but the pressure $P(A)$ above defined does not consider any dynamic. Anyway,  in  \cite{BCMV} and its correction \cite{BCMVZ} there is a study of a general concept of pressure (without a dynamic) in convex analysis, being possible to check that it contains $P(A)$ as an example. 

Our goal is to prove the following theorem.

\begin{theorem} \label{main thm} 
Under above setting, writing $A=-c$ and considering $\beta>0$:

1- there is a pair of functions $\varphi_\beta \in C(X)$ and $\psi_\beta \in C(Y)$ such that 
\[\int  e^{\beta A(x,y)+\varphi_\beta(x)+\psi_\beta(y)}d\mu(x) =1\,\forall y\in Y \,\,\, \text{and}\,\,\, \int  e^{A(x,y)+\varphi_\beta(x)+\psi_\beta(y)}d\nu(y) =1\,\forall x\in X.\]
Such functions are unique in the following sense: if $(\tilde{\phi_\beta},\tilde{\psi_\beta})$ is another pair, then there is a constant $d_\beta$ such that $ 	\tilde{\phi_\beta} = \phi_\beta + d_\beta$ and $\tilde{\psi_\beta} = \psi_\beta - d_\beta$. Furthermore, these functions are Lipschitz continuous and satisfies $\lip(\varphi_\beta)\leq \beta \lip(A)$ and  $\lip(\psi_\beta)\leq \beta \lip(A)$. 

2 - The probability $d\pi_\beta := e^{\beta A+\varphi_\beta+\psi_\beta}d\mu \,d\nu$ belongs to $\Pi(\mu,\nu)$. Furthermore,  
\[P(\beta A) = \smallint  \beta A\,d\pi_\beta +H(\pi_\beta) = -\smallint  \varphi_\beta\,d\mu_\beta - \smallint  \psi_\beta\,d\nu_\beta .\]

3 - The family of functions $(\frac{\varphi_\beta}{\beta})_{\beta>0}$ and $(\frac{\psi_\beta}{\beta})_{\beta>0}$ are equicontinuous and we can suppose uniformly bounded. Any uniform limit $(\varphi, \psi)$ of $(\frac{\varphi_\beta}{\beta}, \frac{\psi_\beta}{\beta})$ as $\beta\to+\infty$ is a solution of the Kantorovich dual problem, that is, it belongs to $\Phi_c$, also verifying  $\alpha(c)=\smallint  \varphi(x)\,d\mu(x)+\smallint  \psi(y)\,d\nu(y)$. Any weak* limit $\pi$ of $\pi_\beta$ as $\beta\to+\infty$ is a solution of the Monge-Kantorovich problem, that is, it belongs to $\Pi(\mu,\nu)$ and satisfies $\smallint  c\,d\pi = \alpha(c)$.
	
\end{theorem}

For compact metric spaces,  the Kantorovich-Rubinstein Theorem claims (see Thm. 1.14 in \cite{Vi1}) that under the case $X=Y$ and $c(x,y)=d(x,y)$ (distance function) we have
\[\alpha(c) = \sup_{\lip(\varphi)\leq 1} \int \varphi\, d(\mu-\nu).\]
   
As a corollary of above theorem we get a solution on the Kantorovich-Rubinstein dual problem. 

\begin{corollary}\label{cor:KR}
Considering the hypotheses and notations of Theorem \ref{main thm} and supposing $X=Y$ and $c(x,y)=d(x,y)$, any uniform limit $\varphi$ of $\frac{\varphi_\beta}{\beta}$ as $\beta\to+\infty$ is a solution of the Kantorovich-Rubinstein dual problem, that is, $\lip(\varphi)\leq 1$ and 
$\alpha(c) = \smallint \varphi \,d(\mu-\nu)$.
\end{corollary} 

Concerning the velocity of convergence of $\pi_\beta$ to the optimal transference plan $\pi$, we study if it satisfies a Large Deviation Principle (see \cite{DZ}).

\begin{proposition}\label{prop:LDP} Under the setting of Theorem \ref{main thm}, supposing there exist the uniform limits $\varphi=\lim_{\beta\to+\infty}\frac{\varphi_{\beta}}{\beta}$, $ \psi=\lim_{\beta\to+\infty}\frac{\psi_{\beta}}{\beta}$ and the weak* limit $\pi=\lim_{\beta\to+\infty}\pi_{\beta}$ we have that $(\pi_{\beta})$ satisfies a large deviation principle with rate function $I(x,y) = c(x,y)-\varphi(x)-\psi(y)$. 
\end{proposition} 	

Naturally we can ask if  known results of Thermodynamic Formalism and Ergodic Optimization could be translated to Transport Theory. Let us denote, for $A=-c$,
\begin{align*} m(A) &:= \sup_{\pi \in \Pi(\mu,\nu)} \smallint  A\, d\pi = - \alpha(c);\\
\mathcal{M}_{max}(A) &:=\{\pi \in \Pi(\mu,\nu)\,|\,\smallint  A\,d\pi =m(A)\};\\
H_{\max} &:= \sup_{\pi \in \mathcal{M}_{max}(A)} H(\pi) \end{align*}
(we can have $H_{\max}=-\infty$).  $\mathcal{M}_{max}(A)$ is the set of optimal plans to $c$ and $H_{\max}$ is the biggest of the entropies of optimal plans.  
Next proposition is similar to a well know result in Thermodynamic Formalism (see \cite{CG}). It presents a characterization of the possible limits of $\pi_\beta$ using the entropy $H$.

\begin{proposition}\label{pressure} 
Under above setting, with $A=-c$, the function $$\beta \mapsto [P(\beta A )-\beta m(A)],\,\,\beta>0$$ is non-increasing and $\lim_{\beta \to+\infty}P(\beta A )-\beta m(A)= H_{\max}$. Furthermore, any accumulation point for $\pi_{\beta}$ as $\beta\to+\infty$ has a maximal entropy $H$ over optimal plans, that is, if $\pi_\infty$ is an accumulation point of $\pi_{\beta}$ as $\beta\to+\infty$, then $\pi_\infty \in \mathcal{M}_{max}(A)$ and $H_{\max}= H(\pi_\infty)$ (which can be $-\infty$).  
\end{proposition}

\section{Basic properties of entropy and pressure}

We remember that  the entropy of a probability $\eta \in \mathcal{P}(X\times Y)$ relative to $\mu \in \mathcal{P}(X)$ was defined as 
$H^\mu(\eta)=-D_{KL}(\eta\,|\, \mu\times \rho)$, where $\rho$ is the $y-$marginal of $\pi$. From Theorems 4.4 and 5.1 of \cite{LM3} we get
\[H^\mu(\eta) = -\sup\{\smallint  u(x,y)d\eta\,|\, \smallint  e^{u(x,y)}d\mu(x)=1\,\forall y,\,\, u \,\text{Lipschitz}\}.\]
If $\pi \in \Pi(\mu,\nu)$ then $\nu$ is the $y-$marginal of $\pi$ and we get 
\begin{equation}\label{eq7}
	 H(\pi) := -D_{KL}(\pi\,|\, \mu\times \nu)=  -\sup\{\smallint  u\,d\pi\,|\, \smallint  e^{u(x,y)}d\mu(x)=1\,\forall y,\,\, u \,\text{Lipschitz}\}.
\end{equation}
 \normalsize

\begin{proposition}\label{prop:entropy} Properties of the entropy $H$ on $\Pi(\mu,\nu)$: \newline
1. $H \leq 0$ and $H(\mu\times \nu)=0$;\newline
2. $H$ is a concave function;\newline 
3. $H$ is upper semi-continuous, that is, if $\pi_n$ converges to $\pi$ in the weak* topology, then $\limsup_n H(\pi_n)\leq H(\pi).$
\end{proposition} 
\begin{proof} Let us denote by $N=\{u:X\times Y \to\mathbb{R}\,|\, \smallint  e^{u(x,y)}d\mu(x)=1\,\forall y,\,\, u \,\text{Lipschitz}\}$. We have, $H(\pi) = \inf_{u\in N} [-\smallint  u\,d\pi]$.
	
1. If we consider $u=0$ we get $u\in N$ and then $H(\pi)\leq 0$. On the other hand, $H(\mu\times \nu) = -D_{KL}(\mu\times \nu|\mu\times \nu)=0$.

2. For $\pi_1,\pi_2 \in \Pi(\mu,\nu)$ and $\lambda \in [0,1]$, if $\pi=\lambda\pi_1+(1-\lambda)\pi_2$ we have
\begin{align*}
	H(\pi) &= \inf_{u\in N} [-\lambda\smallint  u d\pi_1  -(1-\lambda)\smallint  u d\pi_2 ]\\
	       &\geq  \inf_{u_1\in N} [-\lambda\smallint  u_1 d\pi_1]  +\inf_{u_2\in N}[-(1-\lambda)\smallint  u_2 d\pi_2]\\
	       &=\lambda H(\pi_1) +(1-\lambda)H(\pi_2).
\end{align*}  

3. It suppose that $\pi_n \to \pi$ in the weak* topology. \newline
- If $H(\pi)=-\infty$ then for each $k<0$ there is $u\in N$ such that $-\smallint  u\,d\pi \leq k$. As $\pi_n\to\pi$, for sufficiently large $n$ we get $ -\smallint  u\,d\pi_n \leq k+1$, that is, $H(\pi_n)\leq k+1$. Consequently $\limsup_n H(\pi_n) \leq k+1$. As $k$ is arbitrary, $\limsup_n H(\pi_n)=-\infty$.\newline
- If $H(\pi)>-\infty$, for each $\epsilon>0$ there is $u\in N$ such that $-\smallint  u\,d\pi \leq H(\pi)+\epsilon.$  For sufficiently large $n$ we get $ -\smallint  u\,d\pi_n \leq -\smallint  u\,d\pi+\epsilon$, that is, $H(\pi_n) \leq H(\pi)+2\epsilon$. As $\epsilon$ is arbitrary we get $\limsup_n H(\pi_n)\leq H(\pi)$.
\end{proof} 

\begin{example}\label{example:entropy} The entropy $H$ is not affine (see also Example 2.4 in \cite{MO}).\newline
	
Suppose $X=\{1,2\}$, $Y=\{1,2\}$ $\mu=(1/2,1/2)$, $\nu=(1/2,1/2)$ and consider the following probabilities in $\{1,2\}\times\{1,2\}$:
\[\pi_1 = \begin{pmatrix} 1/2&0\\0&1/2\end{pmatrix},\,\, \pi_2 = \begin{pmatrix} 0&1/2\\1/2&0\end{pmatrix},\,\,\pi=	\begin{pmatrix} 1/4&1/4\\1/4&1/4\end{pmatrix}.\]
We have that $\pi,\pi_1,\pi_2 \in \Pi(\mu,\nu)$, $\pi = \frac{1}{2}\pi_1 +\frac{1}{2}\pi_2$, but $H(\pi)=0$ while $H(\pi_1)=H(\pi_2)=-\log(2)$ (consequently $H(\pi) \neq \frac{1}{2}H(\pi_1) +\frac{1}{2}H(\pi_2)$).
Indeed, $\pi = \mu\times\nu$ and then $H(\pi) = 0$. On the other hand, as $X\times Y$ is a finite set, $\pi_1\ll \mu\times \nu$ and its Radon-Nikodym derivative satisfies (a.e.)
$$\frac{d(\pi_1)}{d(\mu\times \nu)} = \begin{pmatrix} 2&0\\0&2\end{pmatrix}=:J_1.$$
We can conclude this, observing that, for any $f:\{1,2\}\times\{1,2\} \to\mathbb{R}$, we have
\[  \smallint  f\,d\pi_1 = f(1,1)\cdot \frac{1}{2}+f(2,2)\cdot\frac{1}{2} =  \sum_{i,j} f(i,j)\cdot J_1(i,j)\cdot\frac{1}{4}= \smallint  f\cdot J_1\, d(\mu\times\nu).\]
Consequently $H(\pi_1) = -\smallint  \log(\frac{d(\pi_1)}{d(\mu\times \nu)})d\pi_1 = -\log(2).$
Similarly  $H(\pi_2) = -\log(2)$.
\end{example}
	
\begin{notation} If $Z$ is a compact metric space, let us denote  the supremum norm on $C(Z)$ by $|f|_{\infty}:=\sup_{z\in Z}|f(z)|$.
	\end{notation}

Next proposition particularly shows that the pressure in the present work is contained in the definition of pressure in \cite{BCMV} and its correction \cite{BCMVZ}.
 	
\begin{proposition} Properties of the pressure \newline
	For any $A,B \in C(X\times Y)$, $c\in \mathbb{R}$ and $\lambda \in[0,1]$ we have:\newline
1. Image: $P(A) \in [\min(A),\max(A)]$;\newline 
2. monotonicity: $A\leq B \Rightarrow P(A) \leq P(B)$;\newline
3. translation invariance: $P(A+c)=P(A)+c$;\newline
4. convexity: $P(\lambda A+(1-\lambda)B)\leq \lambda P(A)+(1-\lambda)P(B)$;\newline
5. continuity: $|P(A)-P(B)|\leq |A-B|_{\infty}$.
\end{proposition}
\begin{proof}
1.	Considering $\pi=\mu\times \nu$ we get 
	\[P(A) \geq \iint A\,d\mu d\nu +H(\mu\times \nu) = \iint A\,d\mu d\nu \geq \min(A).\]
	on the other hand, as $H\leq 0$, for any $\pi\in \Pi(\mu,\nu)$ we have
	\[\smallint  A\,d\pi +H(\pi) \leq \smallint  A\,d\pi \leq \max(A)\] and consequently
	$P( A) \leq \max(A) .$
	
2. and 3.: it is an immediate consequence of  the definition.

4. 
\begin{align*}P(\lambda A+(1-\lambda)B) &= \sup_{\pi} [\lambda \smallint  A\,d\pi + (1-\lambda)\smallint  B\,d\pi + \lambda H(\pi) + (1-\lambda)H(\pi)]\\
	&\leq \sup_{\pi_1} [\lambda (\smallint  A\,d\pi_1 + H(\pi_1))] + \sup_{\pi_2}  [(1-\lambda)(\smallint  B\,d\pi_2 +  H(\pi_2))]\\
	&=\lambda P(A) + (1-\lambda)P(B).
	\end{align*} 

5. Supposing $P(A)\geq P(B)$ we get
\begin{align*}
P(A)-P(B) &= [\sup_{\pi_1} \smallint  A\,d\pi_1 +H(\pi_1)]-[\sup_{\pi_2} \smallint  B\,d\pi_2 +H(\pi_2)]\\
&\leq [\sup_{\pi_1} \smallint  B\,d\pi_1 +H(\pi_1)+|A-B|_\infty]-[\sup_{\pi_2} \smallint  B\,d\pi_2 +H(\pi_2)]\\
&=|A-B|_\infty.
\end{align*} 
	
\end{proof}
	
The set $\Pi(\mu,\nu)$ is convex and we say that $\pi \in \Pi(\mu,\nu)$ is a vertex of $\Pi(\mu,\nu)$ if there is not $\pi_1,\pi_2 \in \Pi(\mu,\nu)$ and $\lambda\in (0,1)$ such that $\pi =\lambda \pi_1 +(1-\lambda)\pi_2$.

\begin{proposition} For each $A\in C(X\times Y)$ there is at least one probability $\pi \in \Pi(\mu,\nu)$ such that $P(A)=\smallint  A\,d\pi + H(\pi)$. It is possible that none of the vertices of $\Pi(\mu,\nu)$ reaches the supremum defining $P(A)$.	
\end{proposition} 		
\begin{proof}
By definition of supremum, for each $n$ there is a probability $\pi_n \in \Pi(\mu,\nu)$ such that 
\[\smallint  A\,d\pi_n + H(\pi_n) \geq P(A) - \frac{1}{n}.\]
As $X\times Y$ is compact there exists a probability $\pi$ and a subsequence $(\pi_{n_j})$ of $(\pi_n)$ such that $\pi_{n_j} \to \pi$ in the weak* topology. We have $\pi \in \Pi(\mu,\nu)$ and as $H$ is upper semi-continuous we get
\[\smallint  A\,d\pi +H(\pi) \geq \limsup_{n_j}[\smallint  A\,d\pi_{n_j} +H(\pi_{n_j})]  = P(A).\]
As the reverse inequality is consequence of the definition of $P$ we get the equality.

Let us present an example where the optimal $\pi$ is unique but it is not a vertex. Consider $X=\{1,2\}$, $Y=\{1,2\}$ $\mu=(1/2,1/2)$, $\nu=(1/2,1/2)$ and the probabilities 
\[\pi_1 = \begin{pmatrix} 1/2&0\\0&1/2\end{pmatrix},\,\, \pi_2 = \begin{pmatrix} 0&1/2\\1/2&0\end{pmatrix},\,\,\pi_3=	\begin{pmatrix} 1/4&1/4\\1/4&1/4\end{pmatrix}.\]
The vertices of $\Pi(\mu,\nu)$ are just the probabilities $\pi_1$ and $\pi_2$. Following Example \ref{example:entropy}, for $A=0$ we get that $P(A) = \sup_{\pi \in \Pi(\mu,\nu)} H(\pi) = 0 = \smallint  A\,d\pi_3 + H(\pi_3)$.
On the other hand, $\smallint  A\,d\pi_1 +H(\pi_1) = -\log(2) = \smallint  A\,d\pi_2 +H(\pi_2)$.  More generally if $A$ is approximately zero the vertices $\pi_1$ and $\pi_2$ are not optimal for the pressure. 
 	
\end{proof}

\section{Proof of  Theorem \ref{main thm}}

\subsection{Proof of item 1}
We can drop $\beta$ in the proof of item 1.

\begin{proposition}\label{prop : existence} There is a pair of functions $\varphi \in C(X)$ and $\psi \in C(Y)$ such that 
	\[\smallint  e^{A(x,y)+\varphi(x)+\psi(y)}d\mu(x) =1\,\forall y\in Y \,\,\, \text{and}\,\,\, \smallint  e^{A(x,y)+\varphi(x)+\psi(y)}d\nu(y) =1\,\forall x\in X.\]
\end{proposition}

\begin{remark}
	If $X=Y$ is a finite set and $\mu=\nu$ is the probability with uniform distribution, we are constructing a double stochastic matrix and such proposition can be found in \cite{R, SK}.
\end{remark}

The proof of this proposition follows some ideas of \cite{bousch} (see also \cite{MO3}).

\begin{proof}  
For each  $0<s<1$, we define  $T_\mu^s: C(X)\to C(Y) $ by
$$T_\mu^s(f)(y) =- \log \int _X e^{A(x,y)+sf(x)}\, d\mu(x)$$
and  $T_\nu^s: C(Y)\to C(X) $ by
$$T_\nu^s(g)(x) =- \log \int _Y e^{A(x,y)+sg(y)}\, d\nu(y).$$

 We have 
\begin{align*}|T_\mu^s(f_1)-T_\mu^s(f_2)|_{\infty}&=\sup_{y} |\log \int  _X e^{A(x,y)+sf_1(x)}\, d\mu(x)- \log \int  _X e^{A(x,y)+sf_2(x)}\, d\mu(x)|\\&\leq s|f_1-f_2|_{\infty}\end{align*}
and analogously
$|T_\nu^s(g_1)-T_\nu^s(g_2)|_{\infty}\leq s|g_1-g_2|_{\infty}.$
It follows that $T_\nu^s\circ T_\mu^s$ is a contraction on $C(X)$ and  $T_\mu^s\circ T_\nu^s$ is a contraction on $C(Y)$ for the supremum norm.
Consequently they have unique fixed points $\varphi^s$ and $\psi^s$, respectively. 

We claim that  $T_\mu^s(\varphi^s)=\psi^s$ and $T_\nu^s(\psi^s)=\varphi^s$. Indeed, denoting $g^s:=T_\mu^s(\varphi^s)$ we get
$T_\mu^s\circ T_\nu^s (g^s) =  T_\mu^s\circ T_\nu^s \circ T_\mu^s(\varphi^s)= T_\mu^s(\varphi^s) = g^s$. As  $\psi^s$ is the unique fixed point of $T_\mu^s\circ T_\nu^s$ we conclude that $g^s=\psi^s$ and then $\psi^s=T_\mu^s(\varphi^s)$. Furthermore
$T_\nu^s(\psi^s) = T_\nu^s(T_\mu^s(\varphi^s))=\varphi^s.$

We claim that for any $s\in (0,1)$, the functions  $\varphi^s$ and $\psi^s$ are Lipschitz continuous, with $\lip(\varphi^s) \leq \lip(A)$ and $\lip(\psi^s) \leq \lip(A)$. Indeed, for any $x_1,x_2 \in X$ we have
\begin{align*} 
	\varphi^s(x_1)  &= -\log\int   e^{A(x_1,y)+s\psi^s(y)}\,d\nu(y) = -\log\int   e^{A(x_2,y) - A(x_2,y) +A(x_1,y)+s\psi^s(y)}\,d\nu(y)\\
	&\leq -\log\int   e^{A(x_2,y) -\lip(A)\cdot d(x_1,x_2)+s\psi^s(y)}\,d\nu(y)\\&=	-\log\int   e^{A(x_2,y) +s\psi^s(y)}\,d\nu(y)+\lip(A)\cdot d(x_1,x_2)\\& = \varphi^s(x_2) +\lip(A)\cdot d(x_1,x_2)\end{align*}
It follows that $ 	\frac{\varphi^s(x_1) - \varphi^s(x_2)}{d(x_1,x_2)}\leq \lip(A)$ for any $x_1,x_2 \in X$. Analogously we can obtain that $\lip(\psi^s) \leq \lip(A)$.

We have that $-\lip(A)\cdot diam(X)\leq \varphi^s-\max(\varphi^s)\leq 0$ and then the family $(\varphi^s-\max(\varphi^s))_{0<s<1}$ is equicontinuous and uniformly bounded and similarly to $(\psi^s-\max(\psi^s))_{0<s<1}$. Before to apply Arzel\`{a}-Ascoli theorem, let us to consider the numbers $\max(\psi^s)+s\max(\varphi^s)$. 

We claim that \begin{equation*}
	-\max(A)\leq \max(\psi^s)+s\max(\varphi^s) \leq lip(A)\cdot diam(X) -min(A).
\end{equation*}
 Indeed, for any $y\in Y$ we have
$e^{-\psi^s(y)} = \int   e^{A(x,y) +s\varphi^s(x)}\,d\mu(x)$ and then $-\psi^s(y)\leq \max(A) + s\max(\varphi^s)$. Consequently $-\max(A) \leq \max(\psi^s) + s\max(\varphi^s).$ On the other hand,
	$$-\psi^s(y) \geq \min(A) + s\min(\varphi^s) \geq \min(A) + s\max(\varphi^s) - s\lip(A)\cdot diam(X) $$
and then $-\min(A) + \lip(A)\cdot diam(X) \geq \max(\psi^s) + s\max(\varphi^s)$. This concludes the proof of the claim. 

With similar arguments we conclude that $\max(\varphi^{s})+s\max(\psi^{s})$ is bounded. Therefore, there exist  an increasing sequence of positive numbers $s_n$  which converges to 1, real numbers  $l_1$, $l_2$ and functions $\varphi_{\beta}\in C(X)$ and $\psi_\beta \in C(Y)$  such that $\max(\psi^{s_n})+s_n\max(\varphi^{s_n})\to l_1$, $\max(\varphi^{s_n})+s_n\max(\psi^{s_n})\to l_2$, $\varphi^{s_n}-\max(\varphi^{s_n}) \to \varphi$ (uniformly) and $\psi^{s_n}-\max(\psi^{s_n}) \to \psi $ (uniformly). Particularly we get $\max(\varphi_\beta)=\max(\psi_\beta)=0$, $\lip(\varphi)\leq \lip(A)$ and $\lip(\psi)\leq \lip(A)$.

We have,
\begin{align*} e^{-\psi(y)} &= \lim_{n}e^{-\psi^{s_n}(y)+\max(\psi^{s_n})} = \lim_n \int   e^{A(x,y) +s_n\varphi^{s_n}(x)}\,d\mu(x) e^{\max(\psi^{s_n})} \\
	&= \lim_n \int   e^{A(x,y) +s_n(\varphi^{s_n}(x) - \max(\varphi^{s_n}))}\,d\mu(x) e^{\max(\psi^{s_n})+s_n\max(\varphi^{s_n})}\\
&=\int   e^{A(x,y) + \varphi(x)}\,d\mu(x)\cdot e^{l_1}.\end{align*}
Consequently
$$\int   e^{A(x,y) + \varphi(x)+\psi(y) +l_1}\,d\mu(x)=1\,\,\forall y.$$
With similar computations we get 
$$\int   e^{A(x,y) + \varphi(x)+\psi(y) +l_2}\,d\nu(y)=1\,\,\forall x.$$ 
Applying Fubini's Theorem we can conclude that
\begin{align*}
	\iint &e^{A(x,y) + \varphi(x)+\psi(y) +l_1}\,d\nu(y)d\mu(x)  = 1= \iint e^{A(x,y) + \varphi(x)+\psi(y) +l_2}\,d\nu(y)d\mu(x).\end{align*}
Then $l_1=l_2$. Finally, replacing $\psi$ by $\psi+l$ where $l:=l_1=l_2$ we complete the proof. As a final remark, observe that by construction, the number $l$ satisfies
\begin{equation}\label{eq10}
	-\max(A) \leq l \leq  lip(A)\cdot diam(X) -min(A)
\end{equation}
\end{proof} 

The proof of next proposition follows ideas of \cite{R, SK}.

\begin{proposition} Suppose there are two pair of continuous functions $(\varphi_1,\psi_1)$, $(\varphi_2,\psi_2)$ such that 
	\begin{equation}\label{eq1} \smallint  e^{A(x,y)+\varphi_i(x)+\psi_i(y)}d\mu(x) =1\,\forall y\in Y \,\,\, \text{and}\,\,\, \smallint  e^{A(x,y)+\varphi_i(x)+\psi_i(y)}d\nu(y) =1\,\forall x\in Y.
		\end{equation} 
	Then there  is a constant $d$ such that $\varphi_2 =\varphi_1+d$ and $\psi_2=\psi_1-d$.
\end{proposition}
\begin{proof}
	Let $p(x) = \varphi_2(x) - \varphi_1(x)$ and $q(y) = \psi_2(y)-\psi_1(y)$. We have
	\begin{equation}\label{eq2} 
				\smallint  e^{A(x,y)+\varphi_1(x)+p(x)+\psi_1(y)+q(y)}d\mu(x) =1\,\forall y\in Y 
			\end{equation}
	and
	\begin{equation}\label{eq3} 
		\smallint  e^{A(x,y)+\varphi_1(x)+p(x)+\psi_1(y)+q(y)}d\nu(y) =1\,\forall x\in X. 
	\end{equation}
Let us denote by $x_0$ and $y_0$ two points satisfying
$p_0 := p(x_0) = \min_{x\in X} p(x)$ and $q_0 := q(y_0) = \max_{y\in Y} q(y).$
For any $y\in Y$ we have
$$e^{q(y)} \stackrel{\eqref{eq2}}{=}\frac{1}{\smallint  e^{A(x,y) +\varphi_1(x)+p(x)+\psi_1(y)}d\mu(x)} \leq \frac{1}{e^{p_0}\smallint  e^{A(x,y) +\varphi_1(x)+\psi_1(y)}d\mu(x)}\stackrel{\eqref{eq1}}{=} e^{-p_0}$$
and for any $x\in X$ we have
$$e^{p(x)} \stackrel{\eqref{eq3}}{=}\frac{1}{\smallint  e^{A(x,y) +\varphi_1(x)+\psi_1(y)+q(y)}d\nu(y)} \geq \frac{1}{e^{q_0}\smallint  e^{A(x,y) +\varphi_1(x)+\psi_1(y)}d\nu(y)}\stackrel{\eqref{eq1}}{=} e^{-q_0}.$$
It follows that $q(y)\leq -p_0$ and $p(x) \geq -q_0$ and then $-p_0 = q_0$. 

Now we will prove that $p$ is constant. Indeed, if $p(\tilde{x}) > p_0$ for some $\tilde{x}$, then, denoting by $\epsilon=\frac{p(\tilde{x})-p_0}{2}>0$, there exists $\delta>0$ such that $p(x)-p_0>\epsilon$ for all $x$ in the ball $B(\tilde{x},\delta)$. It follows that
$$1 \stackrel{\eqref{eq2}}{=} \smallint  e^{A(x,y) +\varphi_1(x)+p(x)+\psi_1(y_0)+q_0}d\mu(x) \stackrel{q_0=-p_0}{=} \smallint  e^{A(x,y) +\varphi_1(x)+\psi_1(y_0)}e^{p(x)-p_0}d\mu(x)$$
$$\geq\smallint  e^{A(x,y) +\varphi_1(x)+\psi_1(y_0)}d\mu(x) + \smallint _{B(\tilde{x},\delta)} e^{A(x,y) +\varphi_1(x)+\psi_1(y_0)}(e^{\epsilon}-1)d\mu(x) \stackrel{\eqref{eq1}}{>}1$$ 
This is a contradiction. Consequently $p(x) = p_0$ for any $x\in X$. Analogously we can show that $q$ is constant. Finally as $
	q_0=-p_0$ we get $q=-p$ and the proof is complete.

\end{proof}

\subsection{Proof of item 2}

Consider $\varphi_\beta$ and $\psi_\beta$ as defined in item 1. Then let us  define the probability $\pi_\beta\in \mathcal{P}(X\times Y)$ by
\[\smallint  u(x,y)\,d\pi_\beta(x,y) := \smallint   e^{\beta A(x,y)+\varphi_\beta(x)+\psi_\beta(y)}u(x,y)d\mu(x) \,d\nu(y),\]
where $u\in C(X\times Y)$.
We have then  $d\pi_\beta := e^{\beta A+\varphi_\beta+\psi_\beta}d\mu \,d\nu$ and consequently $\pi_\beta\ll \mu\times \nu$ with Radon-Nikodym derivative $\frac{d\pi_\beta}{d(\mu\times \nu)}=e^{\beta A+\varphi_\beta+\psi_\beta}$. 
Observe that for any $f\in C(X)$, applying Fubini's Theorem, we have
\begin{align*}
	\smallint  f(x)\,d\pi(x,y) &= \smallint   e^{\beta A(x,y)+\varphi_\beta(x)+\psi_\beta(y)}f(x)d\mu(x) \,d\nu(y)\\& = \smallint   e^{\beta A(x,y)+\varphi_\beta(x)+\psi_\beta(y)}d\nu(y) f(x)\,d\mu(x) \\&
=\smallint   1\cdot f(x)\,d\mu(x) =\smallint   f(x)\,d\mu(x). \end{align*}
Analogously we have $\smallint  g(y)\,d\pi(x,y) = \smallint   g(y)\,d\nu(y) \,\,\forall g\in C(Y) $ and consequently 
 $\pi_\beta \in \Pi(\mu,\nu)$.

By definition of entropy $H$, as $\frac{d\pi_\beta}{d(\mu\times \nu)}=e^{\beta A+\varphi_\beta+\psi_\beta}$ we have 
$$H(\pi_\beta) = -\smallint  \beta A+\varphi_\beta+\psi_\beta \,d\pi_\beta = -\smallint  \beta A d\pi_\beta -\smallint  \varphi_\beta\,d\mu -\smallint  \psi_\beta\,d\nu $$
and then
\begin{equation}\label{eq6}
\smallint  \beta A\,d\pi_\beta +H(\pi_\beta) = -\smallint  \varphi_\beta\,d\mu - \smallint  \psi_\beta\,d\nu .
\end{equation} 
Let us show now that
\begin{equation}\label{eq5}
	P(\beta A) = \sup_{\pi \in \Pi(\mu,\nu)} \smallint  \beta A\,d\pi + H(\pi) = -\smallint  \varphi_\beta\,d\mu - \smallint  \psi_\beta\,d\nu. 
\end{equation}
Let $\pi \in \Pi(\mu,\nu)$ be any transference plan. From equation \eqref{eq7} we have 
\[H(\pi) = \inf\{-\smallint  u(x,y)d\pi\,|\, \smallint  e^{u(x,y)}d\mu(x)=1\,\forall y,\,\, u \,\text{Lipschitz}\}.\]
In this way
$H(\pi) \leq - \smallint  \beta A+\varphi_\beta +\psi_\beta \,d\pi $ and then, as $\pi \in \Pi(\mu,\nu)$,
$$ \smallint  \beta A \,d\pi +H(\pi) \leq - \smallint  \varphi_\beta\,d\mu -\smallint  \psi_\beta\,d\nu.$$
This inequality combined with  \eqref{eq6} concludes the proof of equation \eqref{eq5}.

\subsection{Proof of item 3} 

From item 1. we have $\lip(\varphi_\beta)\leq \lip(\beta A)$ and then $\lip(\frac{\varphi_\beta}{\beta})\leq \lip(A)$. Analogously $\lip(\frac{\psi_\beta}{\beta})\leq \lip(A)$. Furthermore, (following the proof of Proposition \ref{prop : existence}) we can suppose that $-\lip(A)diam(X)\leq \varphi_\beta/\beta\leq 0$ and $-\lip(A)diam(Y)+\frac{l_\beta}{\beta}\leq \psi_\beta/\beta \leq \frac{l_\beta}{\beta}$, where $l_\beta$ satisfies (see equation \eqref{eq10}) 
\[-\beta\max(A) \leq l_\beta \leq \beta \lip(A) diam(X) - \beta \min(A).\]

 From Arzel\`{a}-Ascoli Theorem, there exists an increasing sequence $\beta_n\to+\infty$ and Lipschitz functions $\varphi$ and $\psi$ such that $\frac{\varphi_{\beta_n}}{\beta_n}$ converges uniformly to $\varphi$ and $\frac{\psi_{\beta_n}}{\beta_n}$ converges uniformly to $\psi$.

A proof of the following result can be found in \cite{MO3}.

\begin{lemma}\label{lmms}
	Let $U,Z$ be compact metric spaces. Let $W_{\beta}:U\times Z\to \mathbb{R}$ be a family of measurable functions converging uniformly to a continuous function $W:U\times Z\to \mathbb{R}$, as $\beta\to +\infty,$ and let $\rho$ be a finite measure on $U$ with $\supp(\rho) =U$. Then
	\[\frac{1}{\beta} \log \int _U e^{\beta W_\beta(u,z)} d\rho(u) \to \sup_{u\in U}\, W(u,z)\]  uniformly on $Z,$ as $\beta\to+\infty$.
	The same is true if we replace $\beta$ by a sequence $\beta_n$ which converges to $+\infty$.
\end{lemma}

Using this lemma and taking $\lim_{\beta_n \to+\infty}\frac{1}{\beta_n}\log$ in both sides of the equations
\begin{equation*} \smallint  e^{\beta A(x,y)+\varphi_\beta(x)+\psi_\beta(y)}d\mu(x) =1\,\forall y\in Y \,\,\, \text{and}\,\,\, \smallint  e^{\beta A(x,y)+\varphi_\beta(x)+\psi_\beta(y)}d\nu(y) =1\,\forall x\in X
\end{equation*} 
we get
\begin{equation}\label{eq4} \sup_{x\in X} A(x,y) + \varphi(x)+\psi(y) = 0 \,\forall y\in Y \,\,\text{and} \,\,  \sup_{y\in Y} A(x,y) + \varphi(x)+\psi(y) = 0 \,\forall x\in X. 
\end{equation}
Particularly $\sup_{x,y} A(x,y) + \varphi(x)+\psi(y) = 0$ and then, as $A=-c$ we get $(\varphi,\psi) \in \Phi_c$. 
Furthermore, $(\varphi,\psi)$ is a pair of conjugate $c-$concave functions (see \cite{Vi1} p. 33), that is,
\begin{equation}\label{eq8} 
	\psi(y)=\inf_{x\in X}[ c(x,y) - \varphi(x)] \,\forall y\in Y \,\,\text{and} \,\,  \varphi(x)=\inf_{y\in Y} [c(x,y) -\psi(y)]\,\forall x\in X. 
\end{equation}

If $\pi_\infty$ is any limit of $\pi_\beta$ in the weak* topology (as $\beta \to+\infty$), then clearly $\pi_\infty \in \Pi(\mu,\nu)$. Let us suppose that $\pi_\infty$ is a limit from a subsequence of $\beta_n$, which we also denote by $\beta_n$. Taking $\lim_{\beta_n\to+\infty}\frac{1}{\beta_n}(\cdot)$ in both sides of equation \eqref{eq6} and using that $H$ is non-positive we get
$$\smallint   A\,d\pi_\infty  \geq -\smallint  \varphi\,d\mu - \smallint  \psi\,d\nu, $$ 
and then  (as $A=-c$)
$$
	 \smallint  c\,d\pi_\infty \leq \smallint  \varphi\,d\mu + \smallint  \psi\,d\nu. 
$$
Consequently 
$$\alpha(c) \leq \smallint  c\,d\pi_\infty \leq \smallint  \varphi\,d\mu + \smallint  \psi\,d\nu\leq \alpha(c)$$
( for the last inequality we use that for any $\pi \in \Pi(\mu,\nu)$ we have $  \smallint  \varphi\,d\mu + \smallint  \psi\,d\nu\leq \smallint  c\,d\pi$, because $\varphi+\psi\leq c$). Therefore
$\alpha(c) = \smallint  \varphi\,d\mu + \smallint  \psi\,d\nu = \smallint  c\,d\pi_\infty$. Particularly we conclude that any limit $(\varphi,\psi)$ is an optimal dual pair.

On the other hand, if we start by considering $\pi_\infty$ as any limit of $\pi_{\beta}$, for example taking an increasing sequence $\beta_m$, then from such sequence we  can take a subsequence (also denoted by $\beta_m$) and Lipschitz functions $\varphi$ and $\psi$ such that $\frac{\varphi_{\beta_m}}{\beta_m}$ converges uniformly to $\varphi$ and $\frac{\psi_{\beta_m}}{\beta_m}$ converges uniformly to $\psi$. Finally,  we can reply above computations to reach the same conclusion and then $\pi_\infty$ is an optimal transference plan.

\section{Other proofs} 

\subsection{Proof of Corollary \ref{cor:KR}}

We follow the proof of item 3. of Theorem \ref{main thm} and consider an increasing sequence $\beta_n\to+\infty$ and a  function $\varphi$ such that $\frac{\varphi_{\beta_n}}{\beta_n}$ converges uniformly to $\varphi$. We can take a subsequence of $(\beta_n)$ also denoted by $(\beta_n)$ such that $(\frac{\psi_{\beta_n}}{\beta_n})$ converges uniformly to a function $\psi$. We remark that $\lip(\varphi)\leq \lip (c)$ and  $\lip(\psi)\leq \lip(c)$.

Supposing $X=Y$ and $c(x,y)=d(x,y)$ (distance function) we obtain from equation \eqref{eq8} 
\begin{equation}\label{eq9} 
	\psi(y)=\inf_{x\in X} [d(x,y) - \varphi(x)] \,\forall y\in Y \,\,\text{and} \,\,  \varphi(x)=\inf_{y\in Y} [d(x,y) -\psi(y)]\,\forall x\in X. 
\end{equation}

We remark that $\lip(d)= 1.$ Indeed, for any $(x_1,y_1)\neq (x_2,y_2) \in X\times X$ we have
\begin{align*} \frac{d(x_1,y_1)-d(x_2,y_2)}{d_{X\times Y} ((x_1,y_1),(x_2,y_2))} &= \frac{d(x_1,y_1)-d(x_2,y_2)}{d(x_1,x_2)+d(y_1,y_2)} \leq \frac{d(x_1,x_2)+d(x_2,y_1)-d(x_2,y_2)}{d(x_1,x_2)+d(y_1,y_2)}\\&\leq  \frac{d(x_1,x_2)+d(y_1,y_2)}{d(x_1,x_2)+d(y_1,y_2)}=1 
\end{align*}
and for $x_1\neq x_2=y_1=y_2$ we reach the equality.

Let us show that $\psi=-\varphi$. Analyzing just the function $\psi$, as $\lip(\psi)\leq  1$ we get $\psi(y)-\psi(x) \leq 1\cdot d(x,y) $ for any $x,y$ and consequently $-\psi(x)\leq \inf_y [d(x,y) - \psi(y)]$. As (considering $y=x$ in the infimum)  $-\psi(x)\geq \inf_y [d(x,y) - \psi(y)]$ we conclude that $-\psi(x)= \inf_y [d(x,y) - \psi(y)] \stackrel{\eqref{eq9}}{=}\varphi(x)$.

Applying item 3. of Theorem \ref{main thm} we get
\[\alpha(d)=\alpha(c) = \smallint \varphi\,d\mu + \smallint \psi \,d\nu = \smallint \varphi\,d\mu- \smallint \varphi \,d\nu =\smallint \varphi\,d(\mu-\nu).\]

\subsection{Proof of Proposition \ref{prop:LDP}}
We suppose that there exists the uniform limits $\varphi=\lim_{\beta\to+\infty}\frac{\varphi_{\beta}}{\beta}$, $ \psi=\lim_{\beta\to+\infty}\frac{\psi_{\beta}}{\beta}$ and the weak* limit $\pi=\lim_{\beta\to+\infty}\pi_{\beta}$. We will show that $\pi_{\beta}$ satisfies a large deviation principle with rate function $I(x,y) = c(x,y)-\varphi(x)-\psi(y)$.

\begin{definition}
	Let $(\mu_\beta)_{\beta>0}$ be a family of probabilities on a metric space $\Omega$. We say that $(\mu_\beta)$ satisfy a large deviation principle (LDP) if there exists a lower semi-continuous rate function $I: \Omega \rightarrow[0,+\infty]$ such that\\
	1. 
	$
	\limsup_{\beta \to + \infty } \frac{1}{\beta} \log \mu_\beta(C) \leq-\inf _{\omega \in C} I(\omega),
	$ for any closed set $C \subset \Omega$;\newline
	2. 	$
	\liminf_{\beta \to + \infty } \frac{1}{\beta} \log \mu_\beta(U) \geq-\inf _{\omega \in U} I(\omega),
	$ for any open set $U \subset \Omega$.
\end{definition}

 We refer \cite{DZ}  for general results concerning large deviations. We will apply the following result in this section:

\begin{lemma}\label{lemma: Bric}
	Consider the family $(\pi_\beta)_{\beta>0}$ of probabilities on $X\times Y$. Suppose that for any function $f\in C(X\times Y)$ there exists the limit
	\[\Gamma(f):=\lim_{\beta\to\infty}\frac{1}{\beta}\log (\smallint e^{\beta {f}}\,d\pi_\beta).\]
	Then $(\pi_{\beta})$ satisfies a LDP with rate function
	$I(x,y) = \sup_{x,y}[f(x,y) - \Gamma(f)].$ Furthermore, 
	\begin{equation}\label{eq12} \Gamma(f) = \sup_{x,y}[f(x,y)-I(x,y)],\,\forall \,f\in C(X\times Y).\end{equation}
\end{lemma}
\begin{proof} See Theorem 4.4.2 in  \cite{DZ} (Bric's inverse Varadhan's Lemma).\end{proof}

\begin{remark}
	Using continuous functions of the form
	\[\delta_{(x_0,y_0)}^n (x,y)= \left\{ \begin{array}{cc} -n^2\,[d(x,x_0)+d(y,y_0)]& \text{if}\,\, 0\leq d(x,x_0)+d(y,y_0)< 1/n\\ -n & \text{if}\,\, 1/n \leq d(x,x_0)+d(y,y_0) \end{array}\right.\]
	we can conclude that there is a unique lower semi-continuous function $I$ satisfying equation \eqref{eq12}.
\end{remark}

Now we complete the proof of Proposition \ref{prop:LDP}. We have, by definition of $\pi_\beta$ (see item 2 of Theorem \ref{main thm}  ),
\[\Gamma(f)=\lim_{\beta\to\infty}\frac{1}{\beta}\log (\smallint  e^{\beta A(x,y) + \varphi_{\beta}(x)+\psi_{\beta}(y)+\beta {f(x,y)} }\,d\mu(x)d\nu(y)) .\]
From lemma \ref{lmms} we conclude that $\Gamma(f) = \sup_{x,y} [A(x,y) + \varphi(x) +\psi(y) + f(x,y)]$.   Consequently $(\pi_{\beta})$ satisfies a LDP with (a continuous) rate function $I(x,y) = c(x,y) -\varphi(x)-\psi(y)$.

\subsection{Proof of Proposition \ref{pressure}}

The proof below follows ideas present in \cite{CG}.

\begin{proof}
Let $\pi_\infty$ be an accumulation probability measure of the family $(\pi_{\beta})$ in the weak* topology, as $\beta\to+\infty$. We know from Theorem \ref{main thm} that $\pi_\infty \in M_{max}(A)$.
We have, for $\epsilon>0$,
\begin{align*} P(\beta A) &\geq \beta \smallint  A\, d\pi_{\beta+\epsilon} +H(\pi_{\beta+\epsilon}) = (\beta+\epsilon)\smallint  A\,d\pi_{\beta+\epsilon}+H(\pi_{\beta+\epsilon})-\epsilon\smallint  A\,d\pi_{\beta+\epsilon}\\&=P((\beta+\epsilon)A) - \epsilon\smallint  A\,d\pi_{\beta+\epsilon}.
\end{align*}
Then
\[P((\beta+\epsilon)A) \leq P(\beta A) +\epsilon\smallint  A\,d\pi_{\beta+\epsilon}\]
and therefore
\begin{align*}
	P((\beta+\epsilon)A) &- (\beta+\epsilon)m(A) \leq P(\beta A) +\epsilon\smallint  A\,d\pi_{\beta+\epsilon}- (\beta+\epsilon)m(A)\\
&	= P(\beta A) -\beta m(A) + \epsilon(\smallint  A\,d\pi_{\beta+\epsilon}-m(A))\leq P(\beta A) -\beta m(A).\end{align*}
This shows that $\beta \mapsto [P(\beta A)-\beta m(A)]$ is not increasing. 

We have
\begin{align*} [P(\beta A)-\beta m(A)]=  [H(\pi_{\beta})+\beta(\smallint  A\,d\pi_{\beta}-m(A))]\leq H(\pi_{\beta}).
\end{align*} 
As $[P(\beta A)-\beta m(A)]$ is not increasing and the entropy $H$ is upper semi-continuous, if $\mu_{\beta_i}\to\mu_\infty$ we have
\[\lim_{\beta\to+\infty} [P(\beta A)-\beta m(A)] = \lim_{\beta_i\to+\infty} P(\beta_i A)-\beta_i m(A) \leq \limsup_{\beta_i\to+\infty} H(\pi_{\beta_i}) \leq  H(\pi_\infty).\]

On the other hand, for any $\pi \in \mathcal{M}_{max}(A)$ we have $\beta m(A) = \beta \smallint  A\,d\pi$ and 
\[\lim_{\beta\to+\infty}[P(\beta A)-\beta m(A)] \geq \lim_{\beta\to+\infty}[[\beta \smallint  A\,d\pi+H(\pi)]-\beta \smallint  A\,d\pi] = H(\pi).\]
This shows that $H(\mu_\infty) = H_{\max}$ and that $\lim_{\beta \to+\infty}[P(\beta A )-\beta m(A)]= H_{\max}$.

\end{proof}

\end{document}